\documentclass[nohyperref]{article}

\usepackage{microtype}
\usepackage{graphicx}
\usepackage{subfigure}
\usepackage{booktabs} 

\usepackage{hyperref}

\usepackage[accepted]{icml2022}

\usepackage{amsmath}
\usepackage{amssymb}
\usepackage{mathtools}
\usepackage{amsthm}

\usepackage[capitalize,noabbrev]{cleveref}

\theoremstyle{plain}
\newtheorem{theorem}{Theorem}[section]
\newtheorem{proposition}[theorem]{Proposition}
\newtheorem{lemma}[theorem]{Lemma}

\theoremstyle{definition}

\theoremstyle{remark}

\usepackage[textsize=tiny]{todonotes}

\usepackage[utf8]{inputenc} 
\usepackage[T1]{fontenc}    
\usepackage{hyperref}       
\usepackage{url}            
\usepackage{booktabs}       
\usepackage{amsfonts}       
\usepackage{nicefrac}       
\usepackage{microtype}      
\usepackage{xcolor}         
\usepackage{amsmath}
\usepackage{graphicx}
\usepackage{caption}
\usepackage{amsthm}
\usepackage{amssymb}
\hypersetup{colorlinks, citecolor=blue}

\usepackage{enumitem}

\usepackage{multirow, wrapfig}

\newcommand{\argmin}[1]{\underset{#1}{\mathrm{argmin}}}
\newcommand{\minimize}[1]{\underset{#1}{\mathrm{minimize}}}
\newcommand{\lmo}{\mathrm{LMO}}
\newcommand{\co}{\mathbf{co}}

\newcommand{\diam}{\mathrm{diam}}
\newcommand{\gap}{\mathbf{gap}}

\newcommand{\mD}{\mathcal D}
\newcommand{\mE}{\mathcal E}
\newcommand{\mS}{\mathcal S}

\newcommand{\R}{\mathbb R}
\newcommand{\mb}{\mathbf}

\newcommand{\bs}{\mathbf s}
\newcommand{\bx}{\mathbf x}
\newcommand{\bd}{\mathbf d}

\newcommand{\bz}{\mathbf z}

\newcommand{\bZ}{\mathbf Z}
\newcommand{\bS}{\mathbf S}
\newcommand{\FW}{\textsc{FW}}
\newcommand{\FWflow}{\textsc{FWFlow}}

\newcommand{\sign}{\mathbf{sign}}
\newcommand{\diag}{\mathbf{diag}}

\icmltitlerunning{A Multistep Frank-Wolfe Method}

\begin{document}

\onecolumn
\icmltitle{A Multistep Frank-Wolfe Method}
\icmlsetsymbol{equal}{*}

\begin{icmlauthorlist}
\icmlauthor{ Zhaoyue Chen}{equal,yyy}
\icmlauthor{Yifan Sun}{equal,yyy}
\end{icmlauthorlist}

\icmlaffiliation{yyy}{Stony Brook University}

\icmlcorrespondingauthor{Zhaoyue Chen}{zhaoychen@cs.stonybrook.edu}
\icmlcorrespondingauthor{Yifan Sun}{ysun@cs.stonybrook.edu}

\icmlkeywords{Frank-Wolfe, continuous methods, multistep discretization}

\vskip 0.3in



\printAffiliationsAndNotice{\icmlEqualContribution} 
\begin{abstract}
The Frank-Wolfe algorithm has regained much interest
in its use in structurally constrained machine learning applications. However, one major limitation of the Frank-Wolfe algorithm is the slow local convergence property due to the zig-zagging behavior.
We observe the zig-zagging phenomenon in the Frank-Wolfe method as an artifact of discretization,
and propose multistep Frank-Wolfe variants 
where the truncation errors decay as $O(\Delta^p)$, where $p$ is the method's order. This strategy ``stabilizes" the method, and allows tools like line search and momentum to have more benefit. However, our results suggest that the worst case convergence rate of Runge-Kutta-type discretization schemes cannot improve upon that of the vanilla Frank-Wolfe method for a rate depending on $k$. Still, we believe that this analysis adds to the growing knowledge of flow analysis for optimization methods, and is a cautionary tale on the ultimate usefulness of multistep methods.
\end{abstract}

\section{Introduction}

The Frank-Wolfe method attacks problems of form
\begin{equation}
\minimize{\bx\in \mD}\quad f(\bx)
\label{eq:fw-main}
\end{equation}
where $f:\R^n\to\R$ is an everywhere-differentiable function and $\mD$ is a convex compact constraint set, via the repeated iteration 
\[
\begin{array}{rcl}
\bs^{(k)} &=& \argmin{\bs\in\mD} \; \nabla f(\bx^{(k)})^T\bs\\
\bx^{(k+1)} &=& \gamma^{(k)} \bs^{(k)} + (1-\gamma^{(k)}) \bx^{(k)}.
\end{array}
\tag{\FW}
\]
The first operation is often referred to as the \emph{linear minimization oracle (LMO)}, and is the support function of $\mD$ at $-\nabla f(\bx)$.
In particular, computing the LMO is often computationally cheap, especially when $\mD$ is the level set of a sparsifying norm, e.g. the 1-norm or the nuclear norm. 
In this regime, the advantage of such projection-free methods over methods like projected gradient descent is the cheap per-iteration cost.
However, the tradeoff of the cheap per-iteration rate is that the convergence rate, in terms of number of iterations $k$, is often much slower than that of projected gradient descent \citep{lacoste2015global,freund2016new}. While various acceleration schemes \citep{lacoste2015global} have been proposed and several improved rates given under specific problem geometry \citep{garber2015faster}, by and large the ``vanilla'' Frank-Wolfe method, using the ``well-studied step size'' $\gamma^{(k)} = O(1/k)$, can only be shown to reach $O(1/k)$ convergence rate in terms of objective value decrease \citep{Canon1968ATU,jaggi2013revisiting,freund2016new}

\paragraph{Continuous-time Frank-Wolfe.}
In this work,  we view the method \FW~as an Euler's discretization of the differential inclusion
\[
\begin{array}{rcl}
\dot x(t) &=& \gamma(t)(s(t)-x(t)),\\  s(t) &\in& \argmin{s\in \mD} \nabla f(x(t))^T(s-x(t))
\end{array}
\tag{\FWflow}
\]
where $x(t)$, $s(t)$, and $\gamma(t)$ are continuations of the iterates $\bx^{(k)}$, $\bs^{(k)}$, and coefficients $\gamma^{(k)}$; i.e. $\bx^{(k)} = x( k\Delta )$ for some discretization unit $\Delta$. This was first studied in \citet{jacimovic1999continuous}, and is a part of the construct presented in \citet{diakonikolas2019approximate}. However, neither paper considered the effect of using advanced discretization schemes to better imitate the flow, as a way of improving the method. 
From analyzing this system, we reach the following conclusions through numerical experimentation:
\begin{itemize}[leftmargin=*]
\item (Positive result.) We show that for a class of mixing parameters $\gamma(t)$,  \FWflow~can have an arbitrarily fast convergence rate given  aggressive enough mixing parameters.

\item (Interesting result.)We show qualitatively that, on a number of machine learning tasks, unlike \FW, the iterates $x(t)$ in \FWflow~usually do not zig-zag. 

\end{itemize}

\paragraph{Multistep methods.} While continuous time analyses offer improved intuition in idealized settings, it does not provide a usable method. 
We therefore investigate improved discretization schemes applied to Frank-Wolfe.
Here, we make the following  discoveries:
 \begin{itemize}[leftmargin=*]
 
 \item (Negative result.) We show that over for a particular popular class of multistep methods (Runge Kutta methods) no acceleration can be made when the step size  $\gamma_k = O(1/k)$. 
 
 
 \item (Usefulness.) However, higher order multistep methods tend to have better \emph{search directions}, which accounts for less zig-zagging. This leads to better performance when mixed with line search or momentum methods.
 
\end{itemize}

    
    
    
    

\section{Continuous time Frank-Wolfe}


\begin{proposition}[Continuous flow rate]
\label{prop:fwflow-rate}
Suppose that $\gamma(t) = \frac{c}{c+t}$, for some constant $c \geq 1$. Then the flow rate of \FWflow~ has an upper bound of 
\begin{equation}
\frac{f(x(t))-f^*}{f(x(0))-f^*} \leq  \left(\frac{c}{c+t}\right)^c =  O\left(\frac{1}{t^c}\right).
\label{eq:cont-upper-rate}
\end{equation}
\end{proposition}

Note that this rate  is \emph{arbitrarily fast}, as long as we keep increasing $c$.
This is in stark contrast to the usual convergence rate of the Frank-Wolfe method, which in general \emph{cannot} improve beyond $O(1/k)$ for \emph{any} $c$. 
Figure \ref{fig:limit_continuous} shows this continuous rate as the limiting behavior of \FW, where the discretization steps $\Delta\to 0$. 
\begin{figure*}
    \centering
    \includegraphics[width=4.5in,trim={2.5cm 8cm 2.5cm 0},clip]{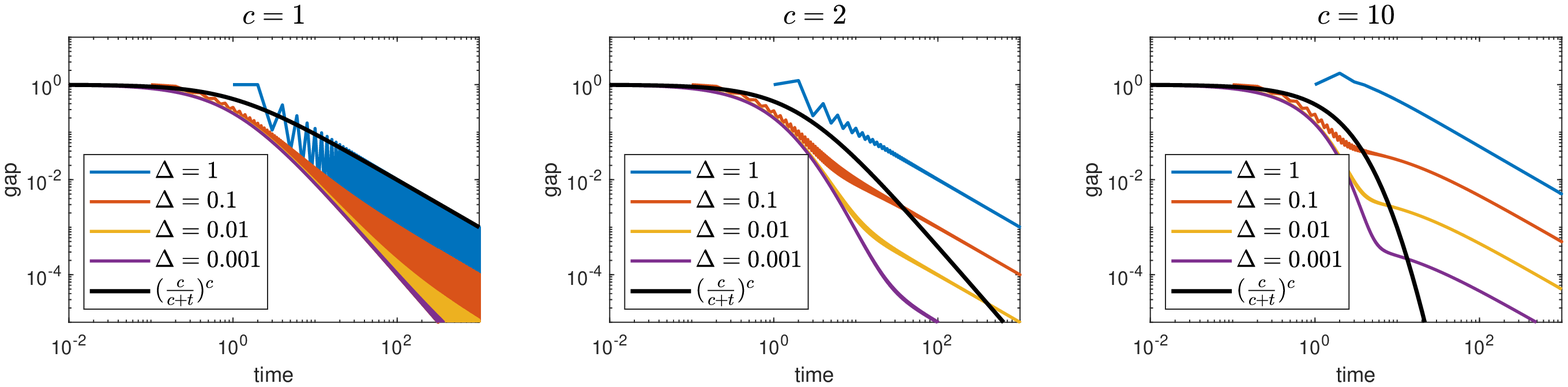}
    \caption{\textbf{Continuous vs discrete.} A comparison of the numerical error vs 
    compared with derived rate. 
    }
    \label{fig:limit_continuous}
\end{figure*}

\subsection{Continuous time Frank Wolfe does not zig-zag}

Figure \ref{fig:continuous_no_zigzag} quantifies this notion more concretely. We first propose to measure ``zig-zagging energy" by averaging the deviation of each iterate's direction across $k$-step directions, for $k = 1,...,W$, for some measurement window $W$:
\[
\mE_{\mathrm{zigzag}}(\bx^{(k+1)},...,\bx^{(k+W)}) = \frac{1}{W-1}\sum_{i=k+1}^{k+W-1} \Big\|\underbrace{\left(I-\frac{1}{\|{\bar\bd}^{(k)}\|_2}\bar\bd^{(k)}({\bar\bd}^{(k)})^T\right)}_{\mb Q}\bd^{(i)}\Big\|_2,
\]
where $\bd^{(i)} = \bx^{(i+1)}-\bx^{(i)}$ is the current iterate direction and $\bar\bd^{(k)} = \bx^{(k+W)}-\bx^{(k)}$ a ``smoothed'' direction. The projection operator $\mb Q$ removes the component of the current direction  in the direction of the smoothed direction, and we measure this ``average deviation energy.''
We divide the trajectory into these window blocks, and report the average of these measurements $\mE_{\mathrm{zigzag}}$ over $T=100$ time steps (total iteration = $T/\Delta$). 
Figure \ref{fig:continuous_no_zigzag} (top table) exactly shows this behavior, where the problem is sparse constrained logistic regression minimization over several machine learning classification datasets~\citep{guyon2004result} (Sensing (ours), Gisette \footnote{Full dataset available at \url{https://archive.ics.uci.edu/ml/datasets/Gisette}. We use a subsampling, as given in 
\url{https://github.com/cyrillewcombettes/boostfw}.
} and Madelon \footnote{Dataset: \url{https://archive.ics.uci.edu/ml/datasets/madelon}})
are shown in Fig. \ref{fig:continuous_no_zigzag}.




From this experiment, we see that any Euler discretization of \FWflow~will always zig-zag, in that the directions often alternate. But, by measuring the deviation across a windowed average, we see two things: first, the deviations converge to 0 at a rate seemingly linear in $\Delta$, suggesting that in the limit as $\Delta \to 0$, the trajectory is smooth. Second, since these numbers are more-or-less robust to windowing size, it suggests that the smoothness of the continuous flow is on a macro level.

\begin{figure*}
    \centering
    \begin{tabular}{lcr}
    \begin{minipage}{.2\textwidth}
    \includegraphics[width=1.5in]{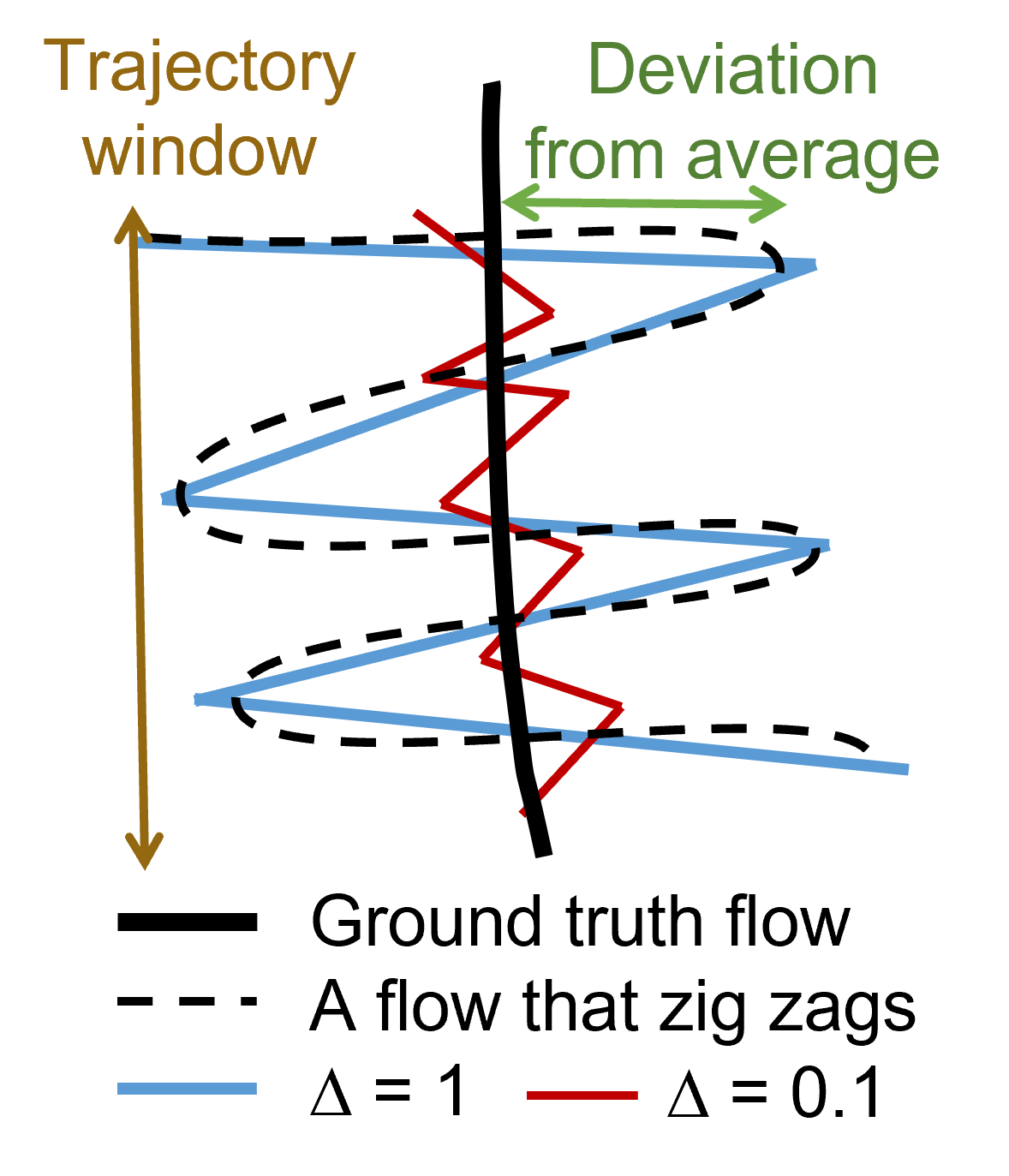}
    \end{minipage}
    &$\quad$ &
    \begin{minipage} {.65\textwidth}
    \begin{tabular}{c}
    
    \begin{minipage}{.8\textwidth}
    \begin{tabular}[t]{c|c|c|c}
    \hline
    Test set&$\Delta = 1$&$\Delta = 0.1$ & $\Delta = 0.01$\\
    \hline
    Sensing & 105.90 / 140.29 & 10.49 / 13.86 & 1.05 / 1.39\\
    Madelon &0.11 / 0.23 &0.021 / 0.028&0.0021 / 0.0028\\
    Gisette &1.08 / 1.74&0.21 / 0.28&0.021 / 0.028\\
    \hline
    \end{tabular}
    \begin{center}
        Zigzagging in continuous flow
    \end{center}
    \end{minipage}
    \\
     \\
    \begin{minipage}{.8\textwidth}
    \begin{tabular}[t]{c|c|c|c}
    \hline
    Test set& FW & FW-MID & FW-RK4\\
    \hline
     Sensing & 105.90 / 140.29 & 0.57 / 1.0018 & 0.015 / 0.033\\
    Madelon & 0.031 / 0.040 & 0.025 / 0.029 &0.025 / 0.029\\
    Gisette & 0.30 / 0.40 & 0.25 / 0.11 & 0.22 / 0.20\\
    \hline
    \end{tabular}
    \begin{center}
        Zigzagging in multistep methods
    \end{center}
    \end{minipage}
    \end{tabular}
    
    \end{minipage}
    
    \end{tabular}
    
    \caption{\textbf{Zig-zagging on real datasets.} 
    Average deviation of different discretizations of \FWflow. Top table uses different $\Delta$s and uses the vanilla Euler's discretization (FW). Bottom uses $\Delta = 1$ and different multistep methods.  
    The two numbers in each box correspond to window sizes 5 / 20.
    }
    \label{fig:continuous_no_zigzag}
\end{figure*}




\section{Runge-Kutta multistep methods}
\label{sec:rkmethod}
\subsection{The generalized Runge-Kutta family}

We now consider Runge-Kutta (RK) methods, a generalized class of higher order methods ($p \geq 1$).
These methods are fully parametrized by some choice of $A\in \R^{q\times q}$, $\beta\in \R^q$, and $\omega\in \R^q$ and at step $k$ can be expressed as (for $i = 1,...,q$)
\begin{eqnarray}
\xi_i &=& \displaystyle\dot x\big(k+\omega_i ,\;  \bx^{(k)}+  \sum_{j=1}^q A_{ij} \xi_j\big),\\
\bx^{(k+1)} &=& \bx^{(k)}+ \sum_{i=1}^q \beta_i \xi_i.
\label{eq:general_discrete}
  \end{eqnarray}
  For consistency,  $\sum_i \beta_i = 1$, and to maintain explicit implementations,  $A$ is always strictly lower triangular. As a starting point, $\omega_1 = 0$.
  Specifically, we refer to the iteration scheme in \eqref{eq:general_discrete} as a \emph{$q$-stage RK discretization method}.
   Note that via our formulation, we capture not just all RK methods, but all explicit multistep methods satisfying this mild consistency constraint, which then implies method feasibility.
A full list of the RK methods used in our experiments is described in the Appendix \ref{app:sec:rk}.

Figure \ref{fig:triangle_momentum} (top row) compares these three implementations on the toy problem,
and shows their rate of convergence. The closeness of the new curves with the continuous flow is apparent; however, while multistep methods are converging faster than vanilla FW, the rate does not seem to change.
However, one thing that \emph{is} visually apparent in Figure \ref{fig:triangle_momentum} (top row)  is that higher order  multistep methods establish better search directions. 
Additionally, we numerically quantify less zig-zagging behavior (lower table in Figure \ref{fig:continuous_no_zigzag}). This is still good news, as there are still several key advantages to such an improvement: namely, better uses of momentum and line search.

\subsection{RK convergence behavior}
All proofs are in the appendix.
\begin{proposition}[Positive result]
All Runge-Kutta methods converge at worst with rate  $f(\bx^{(k)})-f(\bx^*)\leq O(1/k)$.
\label{prop:rungekutta-positive}
\end{proposition}

\begin{proposition}[Negative result]
Under mild conditions, 
 regardless of the order $p$ and choice of $A$, $\beta$, and $\omega$, 
 the worst best case bound for FW-RK, for any RK method, is of order $O(1/k)$.
\label{prop:rungekutta-negative}
\end{proposition}




\section{A better search direction}
Though multistep methods do not seem to improve the rate of convergence, it does improve the quality of the search \emph{direction}.
We leverage this in two ways. 
First, we consider more aggressive line searches, e.g. replacing $\gamma^{(k)}$ with $\max\{\frac{2}{2+k},\bar \gamma\}$ and
\[
\bar\gamma = \max_{0\leq \gamma\leq 1}\{\gamma : f(\bx^{(k)} + \gamma^{(k)} \bd^{(k)}) \leq f(\bx^{(k)})\}.
\]
Note that this is not the typical line search as in \citep{lacoste2015global}, which forces $\gamma^{(k)}$ to be upper bounded by $O(1/k)$--we are hoping for more aggressive, not less, step sizes. 
Second, we follow the scheme presented in \citet{li2020does} which generalizes the 3-variable Nesterov acceleration  \citep{nesterov2003introductory} from gradient descent to Frank-Wolfe. 
Fig. \ref{fig:triangle_momentum}, rows 2 and 3, illustrate the benefits of multistep methods for line search (row 2) and momentum (row 3), over the toy triangle problem. 

\begin{figure*}[!htb]
   \centering
  \includegraphics[height=.8in,trim={2cm 0 2cm 0},clip]{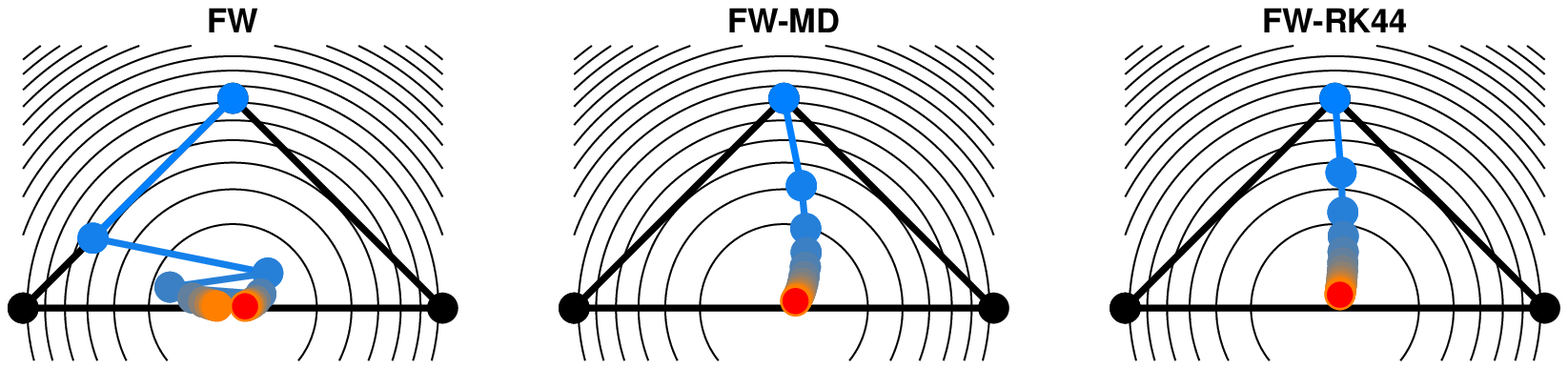}
  \includegraphics[height=.8in]{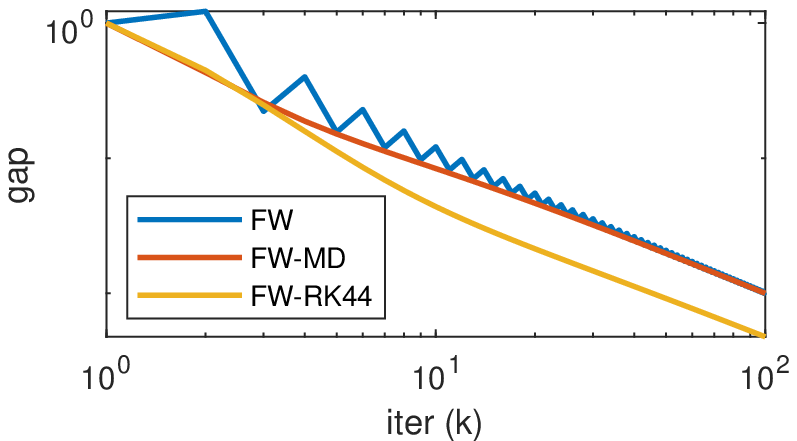}\\

   \centering
  \includegraphics[height=.8in,trim={2cm 0 2cm 0},clip]{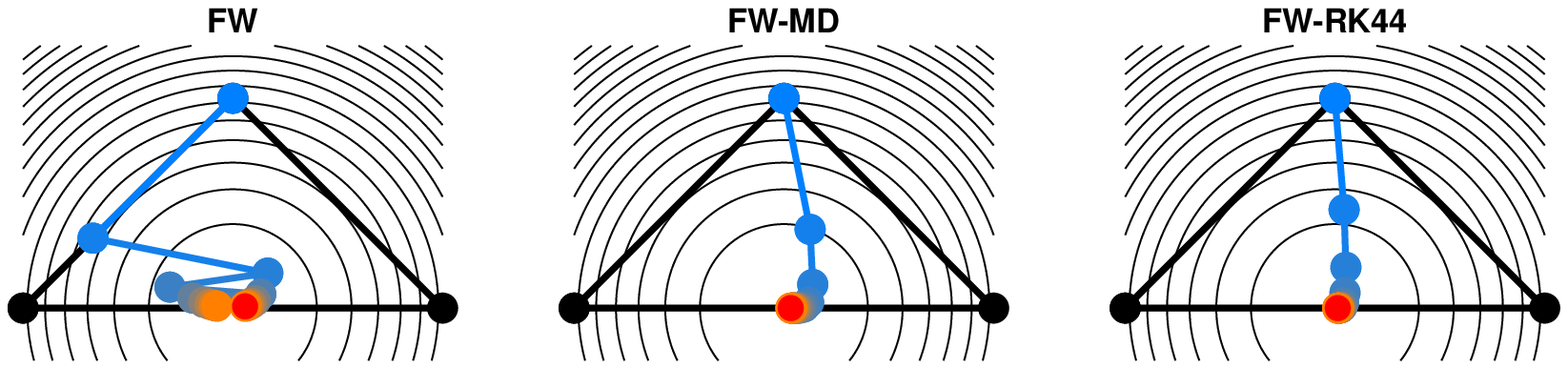}
  \includegraphics[height=.8in]{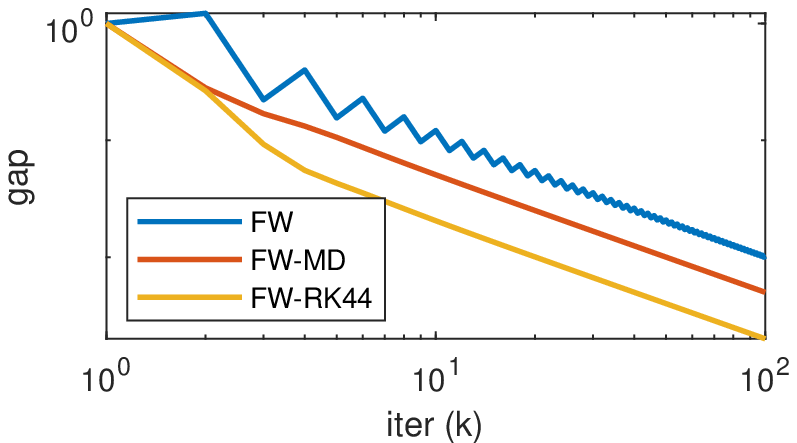}

  \includegraphics[height=.8in,trim={2cm 0 2cm 0},clip]{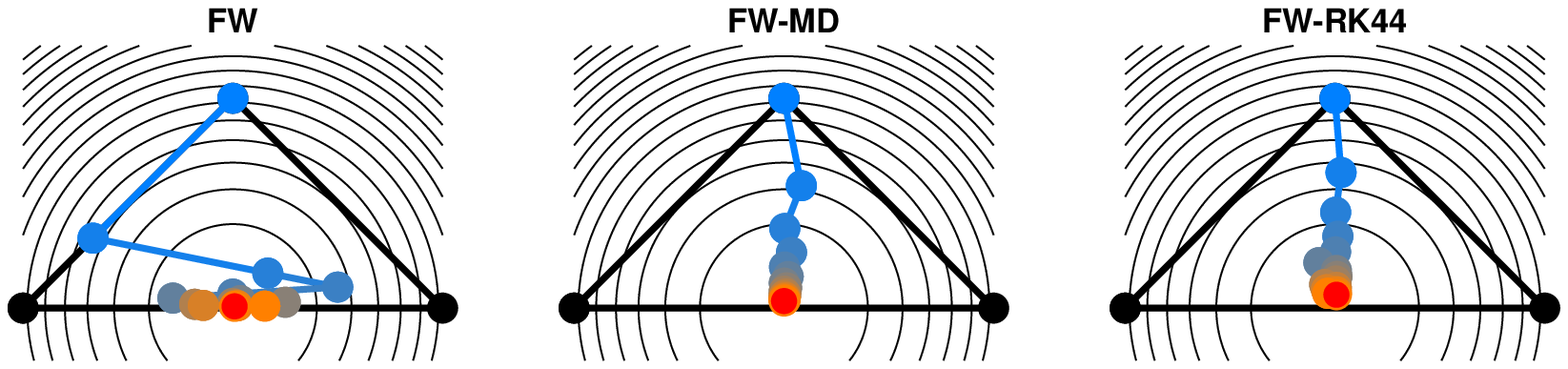}
  \includegraphics[height=.8in]{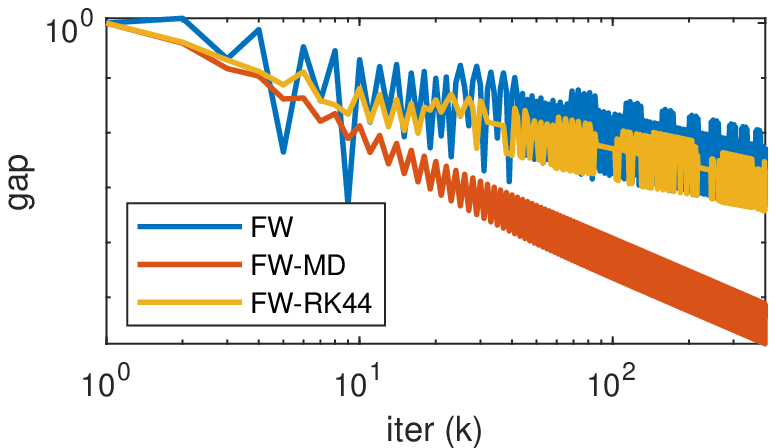}
 \caption{\textbf{Triangle toy problem.}
  \textit{Top.} Straight implementation. \textit{Middle:}  Line search. \textit{Bottom:} Momentum acceleration.}
 \label{fig:triangle_momentum}
\end{figure*}

Finally, we evaluate the benefit of our multistep Frank-Wolfe methods on   sparse logistic regression on Gisette \citep{guyon2004result} (Figure \ref{fig:matfact}, left)
   and  Nuclear-norm constrained Huber regression on  Movielens 100K dataset \citep{harper2015movielens}
\footnote{MovieLens dataset is available at \url{https://grouplens.org/datasets/movielens/100k/}} (Figure \ref{fig:matfact}, right).

\begin{figure}
\centering
  \includegraphics[width=.45\linewidth]{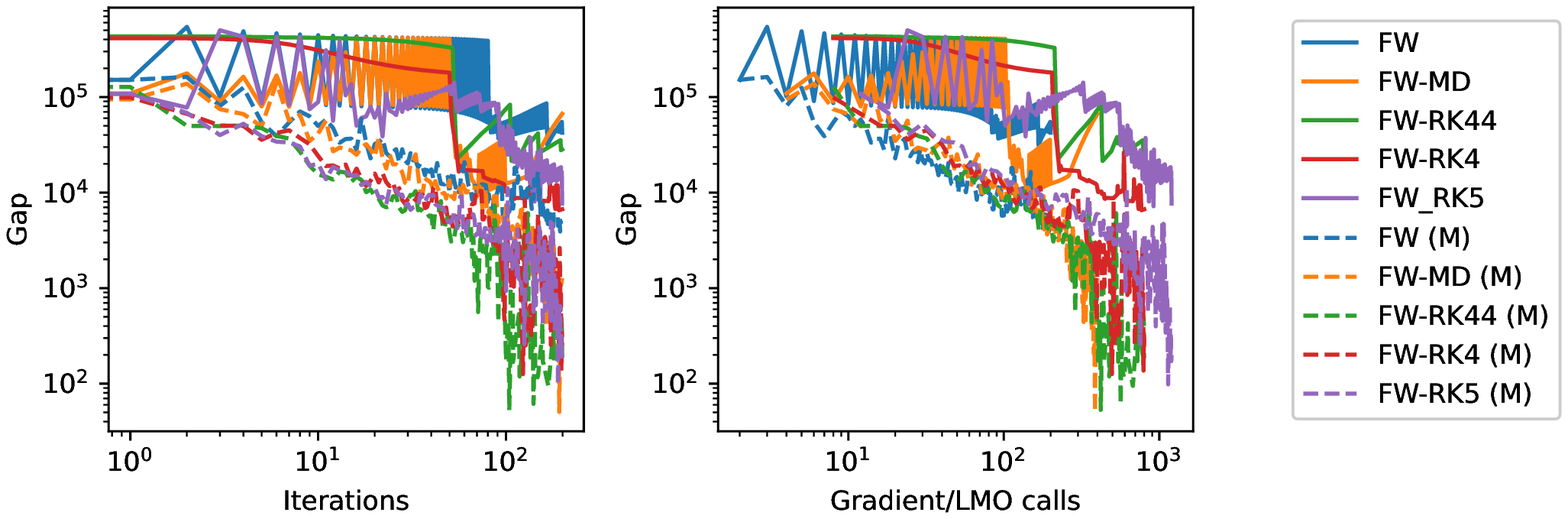}
  \includegraphics[width=.45\linewidth]{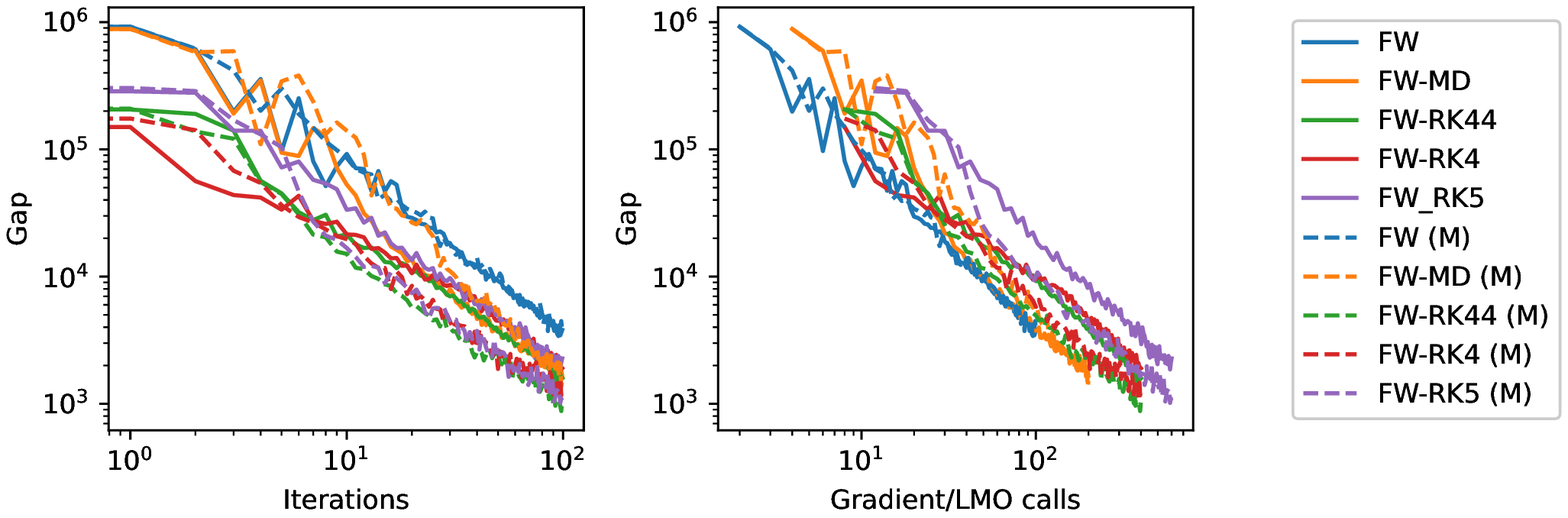}
  \caption{\textbf{Left}: Gisette. \textbf{Right}:Movielens. 
  }
  \label{fig:matfact}
\end{figure}





\newpage

\bibliographystyle{plainnat}
\bibliography{main.bbl}

\newpage
\appendix
\onecolumn

\section{Runge Kutta methods}
\label{app:sec:rk}


\begin{itemize}
    \item Midpoint method
    \[
    A = \begin{bmatrix}
    0 & 0 \\ 1/2 & 0
    \end{bmatrix},
    \qquad
    \beta = \begin{bmatrix}
    0 \\1 
    \end{bmatrix},
    \qquad
    \omega = \begin{bmatrix}
    0\\1/2
    \end{bmatrix},\qquad
    \bz^{(1)} \approx \begin{bmatrix}
    -0.3810\\1.1429
    \end{bmatrix},\qquad
    \bz^{(2)} \approx \begin{bmatrix}
    -0.2222\\0.8889
    \end{bmatrix}
\]
     \item Runge Kutta 4th Order Tableau (44)
    \[
    A = \begin{bmatrix}
   0    & 0    & 0 & 0 \\
     1/2  & 0    & 0 & 0 \\
     0  & 1/2  & 0 & 0 \\
    0    & 0    & 1 & 0 
    \end{bmatrix},
    \qquad
    \beta = \begin{bmatrix}
  1/6   \\ 1/3  \\ 1/3 \\ 1/6
    \end{bmatrix},
    \qquad
    \omega = \begin{bmatrix}
    0   \\
        1/2 \\
        1/2 \\
        1  
    \end{bmatrix},\qquad
    \bz^{(1)} \approx \begin{bmatrix}
    0.2449\\
    0.5986\\
    0.5714\\
    0.3333
    \end{bmatrix}
    \]
    
        \item Runge Kutta 3/8 Rule Tableau (4)
    \[
    A = \begin{bmatrix}
   0    & 0    & 0 & 0 \\
     1/3  & 0    & 0 & 0 \\
     -1/3  & 1  & 0 & 0 \\
     1    & -1    & 1 & 0 \\
    \end{bmatrix},
    \qquad
    \beta = \begin{bmatrix}
 1/8   \\3/8  \\ 3/8 \\ 1/8
    \end{bmatrix},
    \qquad
    \omega = \begin{bmatrix}
   0    \\
        1/3  \\
        2/3 \\
        1  
    \end{bmatrix},\qquad
    \bz^{(1)} \approx \begin{bmatrix}
    0.1758\\
    0.6409\\
    0.6818\\
    0.2500
    \end{bmatrix}
    \]
\item 
Runge Kutta 5 Tableau
\[
    A = \begin{bmatrix}
 0 & 0 & 0 & 0 & 0 & 0\\
 1/4 & 0 & 0 & 0 & 0 & 0\\
  1/8 & 1/8 & 0 & 0 & 0 & 0\\
 0 & -1/2 & 1 & 0 & 0 & 0\\
 3/16 & 0 & 0 & 9/16 & 0 & 0\\
 -3/7 & 2/7 & 12/7 & -12/7 & 8/7 & 0\\
    \end{bmatrix},
    \qquad
    \beta = \begin{bmatrix}
7/90 \\ 0 \\ 32/90 \\ 12/90 \\ 32/90 \\7/90
    \end{bmatrix},
    \qquad
    \omega = \begin{bmatrix}
    0 \\
        1/4 \\
        1/4 \\
        1/2 \\
        3/4\\
        1 
    \end{bmatrix},\qquad
    \bz^{(1)} \approx \begin{bmatrix}
      0.1821\\
    0.0068\\
    0.8416\\
    0.3657\\
    0.9956\\
    0.2333
    \end{bmatrix}
    \]
\end{itemize}
In all examples, $\|\bz^{(k)}\|_\infty$ monotonically decays with $k$.

\section{Continuous Time Frank Wolfe Convergence Rate}
Proof of Prop. \ref{prop:fwflow-rate}
\begin{proof}
\footnote{Much of this proof is standard analysis for continuous time Frank-Wolfe, and is also presented in \cite{jacimovic1999continuous}. }
Note that by construction of $\nabla f(x)^Ts = \displaystyle\min_{y\in\mD}\, \nabla f(x)^Ty$, and since $f$ is convex, 
\[
f(x) - f(x^*) \leq \nabla f(x)^T(x-x^*) \leq \nabla f(x)^T(x-s).
\] 
Quantifying the objective value error as $\mE(t) = f(x(t))-f^*$ 
(where $f^* = \min_{x\in \mD} f(x)$ is attainable) then 
\begin{eqnarray*}
\dot \mE(t) &=& \nabla f(x(t))^T\dot x(t) \\
&\overset{\FWflow}{=}& \gamma(t) \nabla f(x(t))^T(s(t)-x(t)).
\end{eqnarray*}
Therefore, 
\begin{eqnarray*}
\dot \mE(t) &=& -\gamma(t) \underbrace{\nabla f(x(t))^T(x(t)-s(t))}_{\geq f(x)-f(x^*)}\\
&\leq& -\gamma(t) \mE(t) 
\end{eqnarray*}
giving an upper rate of
\[
\mE(t) \leq \mE(0) e^{-\int_0^t\gamma(\tau)d\tau}.
\]
In particular, picking the ``usual step size sequence"  $\gamma(t) = \tfrac{c}{t+c}$ gives the proposed rate \eqref{eq:cont-upper-rate}.
\end{proof}
\section{Feasiblity}
\label{app:sec:feas}

  \begin{proposition}[Feasiblity]
Consider a  $q$-stage multistep FW method
 defined by $A$, $\beta$, and $\omega$. For each given $k\geq 1$, define
  \[
 \bar \gamma_i^{(k)} = \frac{c}{c+k+\omega_i}, \qquad 
 \Gamma^{(k)} = \diag({\bar \gamma^{(k)}}_i),
 \]
 \[
 \mathbf P^{(k)} = \Gamma^{(k)} (I+A^T\Gamma^{(k)})^{-1}, \qquad \mathbf z^{(k)} = q \mathbf P^{(k)}\beta.
\]
Then if $0\leq \mathbf z^{(k)}\leq 1$ for all $k\geq 1$, then 
\[
\bx^{(0)}\in \mD\Rightarrow \bx^{(k)}\in \mD, \quad \forall k\geq 1.
\]
  \end{proposition}
  \begin{proof}

  For a given $k$, construct additionally
  \[
\mathbf Z = \begin{bmatrix} \xi_1 & \xi_2 & \cdots & \xi_q \end{bmatrix},
\]
\[
\bar{\mathbf X} = \begin{bmatrix} \bar \bx_1 & \bar \bx_2 & \cdots & \bar \bx_q \end{bmatrix},\quad
\bar{\mathbf S} = \begin{bmatrix} \bar \bs_1 & \bar \bs_2 & \cdots & \bar \bs_q \end{bmatrix}.
\]
where
      \begin{eqnarray*}
 {\bar \bx}_i &=& \bx^{(k)}+  \sum_{j=1}^q A_{ij} \xi_j, \\
 {\bar \bs}_i &=& \lmo({\bar \bx}_i).
\end{eqnarray*}
  Then we can rewrite \eqref{eq:general_discrete} as
  \begin{eqnarray*}
  \mathbf Z &=& (\bar {\mathbf S} - \bar{\mathbf X})\Gamma = (\bar {\mathbf S} - \bx^{(k)}\mathbf 1^T - \mathbf ZA^T)\Gamma \\
  &=& (\bar {\mathbf S} - \bx^{(k)}\mathbf 1^T )\mathbf P
  \end{eqnarray*}
  for shorthand  $\mathbf P = \mathbf P^{(k)}$ and $\Gamma=\Gamma^{(k)}$. 
 Then
 \begin{eqnarray*}
\bx^{(k+1)} 
&=& \bx^{(k)}(1-\mathbf 1^T \mathbf P\beta) + \bar {\mathbf S} \mathbf P\beta\\
&=&\frac{1}{q}\sum_{i=1}^q  \underbrace{(1-\bz^{(k)}_i)\bx^{(k)} + \bz^{(k)}_i\bar \bs_i }_{\hat\xi_i}
  \end{eqnarray*}
  where $\bz^{(k)}_i$ is the $i$th element of $\bz^{(k)}$, and $\beta = (\beta_1,...,\beta_q)$. 
Then if $0\leq \bz^{(k)}_i\leq 1$, then $\hat \xi_i$ is a convex combination of $\bx^{(k)}$ and $\bar \bs_i$, and $\hat \xi_i\in \mD$ if $\bx^{(k)}\in \mD$. Moreover, $\bx^{(k+1)}$ is an average of $\hat\xi_i$, and thus $\bx^{(k+1)}\in \mD$. Thus we have recursively shown that $\bx^{(k)}\in \mD$ for all $k$.
  \end{proof}

\section{Positive Runge-Kutta convergence result}
\label{app:sec:positiveresults}

\begin{lemma} 
After one step, the generalized Runge-Kutta method satisfies
\[
h(\bx^{(k+1)})-h(\bx^{(k)}) \leq
-\gamma^{(k+1)} h(\bx^{(k)}) + D_4(\gamma^{(k+1)})^2\]

where $h(\bx) = f(\bx) - f(\bx^*)$ and


\[
D_4 =\frac{LD_2^2+2LD_2D_3+2D_3}{2},\quad D_2 = c_1 D, \quad D_3 = c_2c_1 D, \quad c_1 = qp_{\max}, \quad c_2 = q \max_{ij} |A_{ij}|, \quad D = \diam(\mD).
\]

\label{lem:rungekutta-positive-onestep}
\end{lemma}
\begin{proof}
For ease of notation, we write $\bx = \bx^{(k)}$ and $\bx^+ = \bx^{(k+1)}$. We will use $\gamma=\gamma^{(k)} = \tfrac{c}{c+k}$, and $\bar \gamma_i = \tfrac{c}{c+k+\omega_i}$.
Now consider the generalized RK method
\begin{eqnarray*}
\bar \bx_i &=& \bx + \sum_{j=1}^q A_{ij} \xi_j\\
\xi_i &=& \underbrace{\frac{c}{c+k+\omega_i}}_{\tilde \gamma_i} ( \bs_i - \bar {\bx}_i )\\
\bx^+&=&  \bx + \sum_{i=1}^q \beta_i \xi_i\\
\end{eqnarray*}
where $\bs_i = \lmo(\bar \bx_i)$.


Define $D = \diam(\mD)$. 
We use the notation from  section
\ref{sec:rkmethod}. 
Denote the 2,$\infty$-norm as
\[
\|A\|_{2,\infty} = \max_j \|a_j\|_2
\]
where $a_j$ is the $j$th column of $A$. Note that all the element-wise elements in 
\[
\mathbf P^{(k)} = \Gamma^{(k)}(I+A^T\Gamma^{(k)})^{-1}
\]
is a decaying function of $k$, and thus defining $p_{\max} = \|\mathbf P^{(1)}\|_{2,\infty}$
we see that
\[
\|\bar {\mathbf Z}\|_{2,\infty} = \|(\bar {\mathbf S} - \bx^{(k)}\mathbf 1)\mathbf P^{(k)}\|_{2,\infty} \leq qp_{\max} D.
\]

Therefore, since $\bar {\mathbf Z} = (\bar {\mathbf S}-\bar {\mathbf X})\Gamma$, and all the diagonal elements of $\Gamma$ are at most 1, 
\[
\|\bs_i-\bar \bx_i\|_2  \leq qp_{\max} D =: D_2
\]
and
\[
\|\bx-\bar \bx_i\|_2 = \|\sum_{j=1}^q A_{ij} \gamma_j (\bs_j-\bar \bx_j)\|_2 \leq q \max_{ij} |A_{ij}| \gamma D_2 =: D_3 \gamma.
\]

Then 
\begin{eqnarray*}
f(\bx^+)-f(\bx) &\leq&  \nabla f(\bx)^T(\bx^+-\bx) + \frac{L}{2}\|\bx^+-\bx\|_2^2\\
&=&  \sum_i \beta_i \tilde \gamma_i \nabla f(\bx)^T(\bs_i-\bar \bx_i) + \frac{L}{2}\underbrace{\|\sum_i \beta_i \tilde \gamma_i (\bs_i-\bar \bx_i)\|_2^2}_{\leq \gamma^2 D_2^2}\\
&=&  \sum_i \beta_i \tilde \gamma_i (\nabla f(\bx)-\nabla f(\bar \bx_i))^T(\bs_i-\bar \bx_i) +
 \sum_i \beta_i \tilde \gamma_i \underbrace{\nabla f(\bar \bx_i)^T(\bs_i-\bar \bx_i)}_{-\gap(\bar \bx_i)} +
 \frac{L\gamma^2D_2^2}{2}\\
 &\leq& \sum_i \beta_i\underbrace{ \tilde \gamma_i}_{\leq \gamma} \underbrace{\|\nabla f(\bx)-\nabla f(\bar \bx_i)\|_2}_{L\|\bx-\bar \bx_i\|_2=L\gamma D_3}\underbrace{\|\bs_i-\bar \bx_i\|_2}_{\leq D_2} - \sum_i \beta_i\tilde\gamma_i \gap(\bar \bx_i) + \frac{L\gamma^2 D_2^2}{2}\\
 &\leq &  -\sum_i \beta_i\tilde\gamma_i \gap(\bar \bx_i) + \frac{L\gamma^2 D_2^2}{2} + \frac{2L\gamma^2 D_2D_3}{2}\\ 
 &\leq& -\gamma^+ \sum_i\beta_i h(\bar \bx_i)  + \frac{L\gamma^2D_2(D_2+2D_3)}{2} 
\end{eqnarray*}
where $\gamma=\gamma_k$, and $\gamma^+=\gamma_{k+1}$. Now assume $f$ is also $L_2$-continuous, e.g. $|f(\bx_1)-f(\bx_2)|\leq L_2\|\bx_1-\bx_2\|_2$. Then, taking  $h(\bx) = f(\bx) -f(\bx^*)$,

\begin{eqnarray*}
h(\bx^+)-h(\bx) 
 &\leq& -\gamma^+ \sum_i\beta_i (h(\bar \bx_i)-h(\bx)) -\gamma^+ \underbrace{\sum_i\beta_i}_{=1} h(\bx) + \frac{L\gamma^2D_2(D_2+2D_3)}{2}\\
 &\leq & 
 \gamma \sum_i\beta_iL_2 \underbrace{\|\bar \bx_i-\bx\|_2}_{\leq \gamma D_3}-\gamma^+ h(\bx) + \frac{L\gamma^2D_2(D_2+2D_3)}{2}\\
 &\leq & -\gamma^+ h(\bx) + \frac{\gamma^2(LD_2^2+2LD_2D_3+2D_3)}{2}\\
 &\leq& -\gamma^+ h(\bx) + D_4(\gamma^+)^2
\end{eqnarray*}
where 
$D_4 =\frac{LD_2^2+2LD_2D_3+2D_3}{2}$and we use $2 \geq (\gamma/\gamma^+)^2$ for all $k \geq 1$.

\end{proof}

Proof of Prop. \ref{prop:rungekutta-positive}
\begin{proof}
After establishing Lemma \ref{lem:rungekutta-positive-onestep}, the rest of the proof is a recursive argument, almost identical to that in  \cite{jaggi2013revisiting}. 

At $k = 0$, we define $h_0 =\max\{ h(\bx^{(0)}), \frac{ D_4c^2}{c-1}\}$, 
and it is clear that $h(\bx^{(0)}) \leq  h_0$.

Now suppose that for some $k$, $h(\bx^{(k)}) \leq \frac{h_0}{k+1}$. Then
\begin{eqnarray*}
h(x_{k+1}) &\leq &  h(\bx_k) - \gamma_{k+1}h(\bx^{(k)}) + {D_4} \gamma_{k+1}^2\\
&\leq & \frac{h_0}{k+1}\cdot \frac{k+1}{c+k+1} + D_4 \frac{c^2}{(c+k+1)^2}\\
&=& \frac{h_0}{c+k+1} + D_4 \frac{c^2}{(c+k+1)^2}\\
&=& \left( h_0 + \frac{D_4c^2}{c+k+1}\right) \left(\frac{k+2}{c+k+1}\right) \frac{1}{k+2}
\\
&\leq& h_0\left( 1+\frac{c-1}{c+k+1}\right) \left(\frac{k+2}{c+k+1}\right) \frac{1}{k+2}\\
\\
&\leq& h_0\underbrace{\left( \frac{2c+ k }{c+k+1}\right) \left(\frac{k+2}{c+k+1}\right)}_{\leq 1} \frac{1}{k+2}.
\end{eqnarray*}
\end{proof}

\section{Negative Runge-Kutta convergence result}
\label{app:sec:negativeresults}

This section gives the proof for Proposition \ref{prop:rungekutta-negative}.

\begin{lemma}[$O(1/k)$ rate]\label{lem:o1krate}
Start with $\bx^{(0)} = 1$. Then consider the sequence defined by
\[
\bx^{(k+1)}= |\bx^{(k)} - \frac{c_k}{k}|
\]
where, no matter how large $k$ is, there exist some constant where  $C_1 < \max_{k'>k} c_{k'} $.
(That is, although $c_k$ can be anything, the smallest upper bound of $c_k$ does not decay.) Then
\[
\sup_{k'\geq k} |\bx^{(k')}| = \Omega(1/k).
\]
That is, the smallest upper bound of $|\bx^{(k)}|$ at least of order $1/k$.
\end{lemma}
\begin{proof}
We will show that the smallest upper bound of $|\bx^{(k)}|$ is larger than $C_1/(2k)$.

Proof by contradiction. 
Suppose that at some point $K$, for all $k \geq K$,  $|\bx^{(k)}| < C_1/(2k)$. Then from that point forward, 
\[
\sign(\bx^{(k)}-\frac{c_k}{k}) = -\sign(\bx^{(k)})
\]
and there exists some $k' > k$ where $c_{k'} > C_1$. Therefore, at that point,
\[
|\bx^{(k'+1)}| = \frac{c_{k'}}{k'}-|\bx^{(k')}| 
\geq \frac{C_1}{2k'}>\frac{C_1}{2(k'+1)}.
\]
This immediately establishes a contradiction.
\end{proof}

Now define the operator 
\[
T(\bx^{(k)}) = \bx^{(k+1)}-\bx^{(k)}
\]
and note that 
\[
|\bx^{(k+1)}| = |\bx^{(k)}+T(\bx^{(k)})| = | |\bx_k|+\sign(\bx^{(k)})T(\bx^{(k)})|.
\]
Thus, if we can show that there exist some $\epsilon$, agnostic to $k$ (but possibly related to Runge Kutta design parameters), and
\begin{equation}
\exists k'\geq k, \quad -\sign(\bx^{(k')})T(\bx^{(k')}) > \frac{\epsilon}{k'},\quad \forall k,
\label{eq:lemma-helper-1}
\end{equation}
 then based on the previous lemma, this shows $\sup_{k'>k}|\bx_{k'}| = \Omega (1/k)$ as the smallest possible upper bound.

\begin{lemma}
Assuming that $0 <q\mb P^{(k)} \beta < 1$ then there exists a finite point $\tilde k$ where for all $k > \tilde k$, 
\[
|\bx^{(k)}| \leq \frac{C_2}{k}
\]
for some $C_2 \geq 0$.
\end{lemma}
\begin{proof}

We again use the block matrix notation
\[
\bZ^{(k)} = (\bar \bS-\bx^{(k)}\mb 1^T) \Gamma^{(k)}(I+A^T\Gamma^{(k)})^{-1}
\]
where $\Gamma^{(k)} = \diag(\tilde \gamma_i^{(k)})$ and each element $\tilde \gamma_i^{(k)} \leq \gamma^{(k)}$.

First, note that by construction, since 
\[
\|\bar \bS-\bx^{(k)}\mb 1^T\|_{2,\infty} \leq D_4, \quad \|(I+A^T\Gamma^{(k)})^{-1}\|_2 \leq  \|(I+A^T\Gamma^{(0)})^{-1}\|_2
\]
are bounded above by constants, then 
\[
\|\bZ^{(k)}\|_\infty \leq \frac{c}{c+k} C_1
\]
for $C_1 = D_4\|(I+A^T\Gamma^{(0)})^{-1}\|_2 $.

First find constants $C_3 $, $C_4$, and $\bar k$ such that
\begin{equation}
\frac{C_3}{k} \leq \mb 1^T \mathbf P^{(k)} \beta \leq \frac{C_4}{k}, \quad \forall k>\bar k,
\label{eq:boundx_helper}
\end{equation}
and such constants always exist, since
by assumption, there exists some $a_{\min} > 0$, $a_{\max}<1$ and some $k'$ where
\[
a_{\min} <q\mb P^{(k')} \beta < a_{\max} \Rightarrow \frac{a_{\min}}{q \gamma_{\max}} \leq (I+A^T\Gamma^{(k')})^{-1} \beta \leq \frac{a_{\max}}{q \gamma_{\min}}
\]
where 
\[
\gamma_{\min} = \min_i \frac{c}{c+k'+\omega^{(k')}_i}, \qquad \gamma_{\max} = \frac{c}{c+k'}.
\]
Additionally, for all $k > c+1$,
\[
\frac{c}{2k}\leq \frac{c}{c+k+1} \leq \Gamma^{(k)}_{ii} \leq \frac{c}{c+k} \leq \frac{c}{k}.
\]
Therefore taking
\[
C_3 = \frac{ca_{\min}}{2q \gamma_{\max} }, \qquad C_4 = \frac{c a_{\max}}{q\gamma_{\min}}, \qquad \bar k = \max\{k',c+1\}
\]
satisfies \eqref{eq:boundx_helper}.

Now define
\[
C_2 = \max\{|\bx^{(1)}|,4cq C_1 \|A\|_\infty, 4C_3, 4C_4\}.
\]
We will now inductively show that $|\bx^{(k)}|\leq \frac{C_2}{k}$. From the definition of $C_2$, we have the base case for $k = 1$:
\[
|\bx^{(1)}| \leq \frac{|\bx^{(1)}|}{1} \leq \frac{C_2}{k}.
\]
Now assume that $|\bx^{(k)}|\leq \frac{C_2}{k}$. Recall that
\[
\bx^{(k+1)} = \bx^{(k)}(1-\mb 1^T \mathbf P^{(k)} \beta) + \bar \bS\mathbf P^{(k)} \beta, \qquad \bar \bS = [\bar \bs_1,...,\bar\bs_q], \qquad \bs_i = -\sign(\bar \bx_i)
\]
and we denote the composite mixing term $\bar \gamma^{(k)} = \mb 1^T\mb P^{(k)} \beta$. 
We now look at two cases separately.
\begin{itemize}
    \item Suppose first that $\bar \bS = -\sign(\bx^{(k)}\mb 1^T)$, e.g. $\sign(\bar\bx_i) = \sign(\bx^{(k)})$ for all $i$. Then 
    \[
    \bar \bS \mb P^{(k)} \beta = -\sign(\bx^{(k)})\bar \gamma_k,
    \]
    and 
    \begin{eqnarray*}
    |\bx^{(k+1)}| &=&  |\bx^{(k)}(1-\bar\gamma^{(k)}) + \bar \bS\mathbf P^{(k)} \beta|\\
    &=&  |\bx^{(k)}(1-\bar\gamma^{(k)}) -\sign(\bx^{(k)})\bar\gamma^{(k)}|\\
    &=&  |\underbrace{\sign(\bx^{(k)})\bx^{(k)}}_{|\bx^{(k)}|}(1-\bar\gamma^{(k)}) -\underbrace{\sign(\bx^{(k)})\sign(\bx^{(k)})}_{=1}\bar\gamma^{(k)}|\\
    &=&  | |\bx^{(k)}|(1-\bar\gamma^{(k)}) -\bar\gamma^{(k)}|\\
    &\leq&  \max\{ |\bx^{(k)}|(1-\bar\gamma^{(k)}) -\bar\gamma^{(k)},
    \bar\gamma^{(k)} - |\bx^{(k)}|(1-\bar\gamma^{(k)}) 
    \}\\
    &\leq&  \max\Bigg\{ \underbrace{\frac{C_2}{k}(1-\frac{C_3}{k}) -\frac{C_3}{k}}_{(*)},
    \frac{C_4}{k} \Bigg\}\\
    \end{eqnarray*}
    and when $k \geq \frac{C_2}{C_3} \iff C_3 \geq \frac{C_2}{k}$,
    \[
    (*) \leq C_2\left(\frac{1}{k} -\frac{1}{k^2}\right) \leq \frac{C_2}{k+1}.
    \]
    Taking also $C_4 \leq \frac{C_2}{4}$,
    \begin{eqnarray*}
    |\bx^{(k+1)}| 
    \leq  \max\left\{ \frac{C_2}{k+1},
    \frac{C_2}{4k} \right\} \leq \frac{C_2}{k+1}\\
    \end{eqnarray*}
    for all $k \geq 1$.
    
    \item Now suppose that there is some $i$ where $\bar\bs_i = \sign(\bx^{(k)}\mb1^T)$. 
    Now since 
\[
\bar \bS = -\sign(\bx^{(k)}\mb 1^T + \bZ A^T)
\]
then this implies that 
$ |\bx^{(k)}| < (\bZ A^T)_i$.
But since 
\[
|(\bZ A^T)_i| \leq \|\bZ\|_\infty \|A\|_\infty q \leq \frac{c}{c+k}(C_1\|A\|_\infty q) \leq \frac{C_2}{4k}, 
\]
this implies that
\[
|\bx^{(k+1)}| \leq \frac{C_2}{4k}(1-\frac{C_3}{k}) + \frac{C_2}{4k} \leq \frac{C_2}{2k}\leq \frac{C_2}{k+1}, \quad \forall k > 1.
\]
\end{itemize}
Thus we have shown the induction step, which completes the proof.
\end{proof}

\begin{lemma}

There exists a finite point $\tilde k$ where for all $k > \tilde k$, 
\[
 \frac{c}{c+k}-\frac{C_4}{k^2} < |\xi_i| < \frac{c}{c+k}+\frac{C_4}{k^2}
\]
for some constant $C_4>0$.
\end{lemma}
\begin{proof}
Our goal is to show that 
\[
\gamma^{(k)} - \frac{C_4}{k^2} \leq \|\bZ\|_\infty \leq  \gamma^{(k)} + \frac{C_4}{k^2}
\]
for some $C_4\geq 0$, and for all $k \geq k'$ for some $k' \geq 0$.
Using the Woodbury matrix identity,
\[
\Gamma(I+A^T\Gamma)^{-1} = \Gamma\left(I - A^T(I+\Gamma A^T)^{-1} \Gamma\right)
\]
and thus
\[
\bZ^{(k)} = \bar \bS\Gamma -\underbrace{\left(\bx^{(k)}\mb 1^T \Gamma + (\bar \bS-\bx^{(k)}\mb 1^T)\Gamma A^T(I+\Gamma A^T)^{-1}\Gamma\right)}_{\mathbf B}.
\]
and thus
\[
   |\bar \bs_i \tilde \gamma_i| - \frac{C_3}{k^2} \leq |\xi_i^{(k)}| \leq |\bar \bs_i \tilde \gamma_i| + \frac{C_3}{k^2}
\]
where via triangle inequalities and norm decompositions,
\[
\frac{C_3}{k^2} = \max_i |\mathbf B_i| \leq \underbrace{|\bx^{(k)}|}_{O(1/k)} \gamma_k + D_4 \gamma_k^2 \|A\|_\infty (I+\Gamma^{(0)}A^T)^{-1} = O(1/k^2).
\]
Finally, since $\bar \bs_i\in \{-1,1\}$, then 
$|\bar \bs_i\tilde \gamma_i| = \tilde\gamma_i$, and in particular,
\[
\frac{c}{c+k+\omega_i}\leq \frac{c}{c+k}
\]
and
\[
\frac{c}{c+k+\omega_i}\geq \frac{c}{c+k+\omega_{\max}} = \frac{c}{c+k} - \frac{c}{c+k}\frac{\omega_{\max}}{c+k+\omega_{\max}} \geq \frac{c}{c+k}-\frac{c\omega_{\max}}{k^2}
\]
Therefore, taking $C_4 = c\omega_{\max}+C_3$ completes the proof.

\end{proof}

\begin{lemma}
There exists some large enough $\tilde k$ where for all $k \geq \tilde k$, it must be that
\begin{equation}
\exists k'\geq k, \quad -\sign(\bx^{(k')})T(\bx^{(k')}) > \frac{\epsilon}{k'}.
\label{eq:lemma-helper-1}
\end{equation}
\end{lemma}

\begin{proof}
Define a partitioning $S_1\cup S_2 = \{1,...,q\}$, where
\[
S_1 = \{i : \xi_i > 0\}, \quad S_2 = \{j : \xi_j\leq 0\}.
\]
Defining $\bar \xi = \frac{c}{c+k}$,
\[
|\sum_{i=1}^q \beta_i \xi_i| = |\sum_{i\in S_1}\beta_i |\xi_i| -\sum_{j\in S_2}\beta_j|\xi_j|| \geq  \left(\bar \xi-\frac{C_4}{k^2}\right)\cdot\left|\sum_{i\in S_1}\beta_i-\sum_{j\in S_2}\beta_j\right|.
\]

By assumption, there does not exist a combination of $\beta_i$ where a specific linear combination could cancel them out; that is, suppose that there exists some constant $\bar \beta$, where for \emph{every} partition of sets $S_1$,$S_2$,
\[
0<\bar\beta :=\min_{S_1,S_2}  |\sum_{i\in S_1}\beta_i-\sum_{j\in S_2}\beta_j|.
\]
Then  
\[
|\sum_{i=1}^q \beta_i \xi_i|  \geq  \left(\frac{c}{c+k}-\frac{C_2}{k^2}\right)\bar\beta \geq \bar\beta\frac{\max\{C_2,c\}}{k}.
\]
Picking $\epsilon = \max\{C_2,c\}$ concludes the proof.
\end{proof}

\section{More Higher Order Discretization Methods}
Figure \ref{fig:more_quadratic} evaluates the performance of more multistep Frank-Wolfe methods, for a problem with  $m = 500$, $n = 100$, and $\alpha = 1000$.
\begin{figure}[!htb]
  \includegraphics[height=2in]{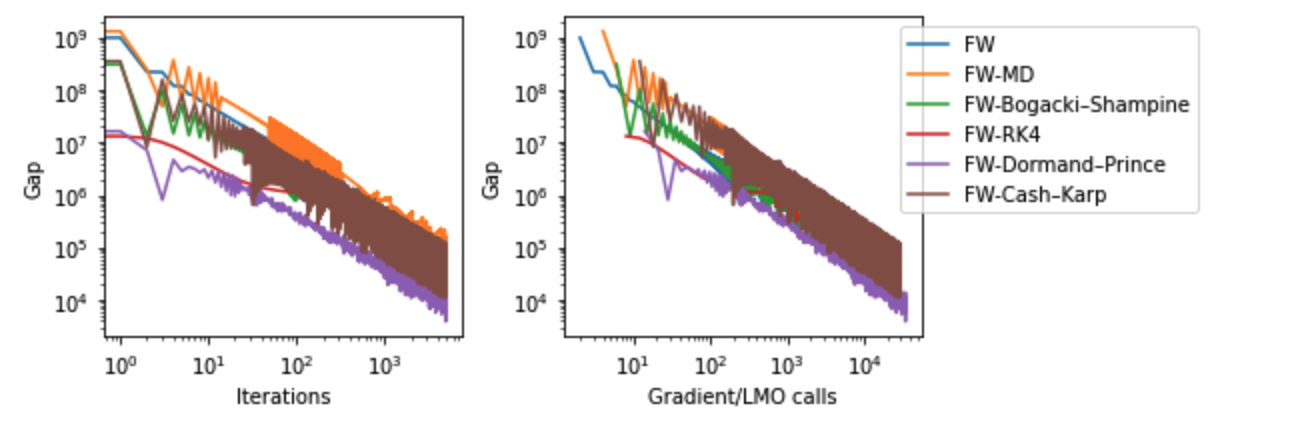}
  \includegraphics[height=2in]{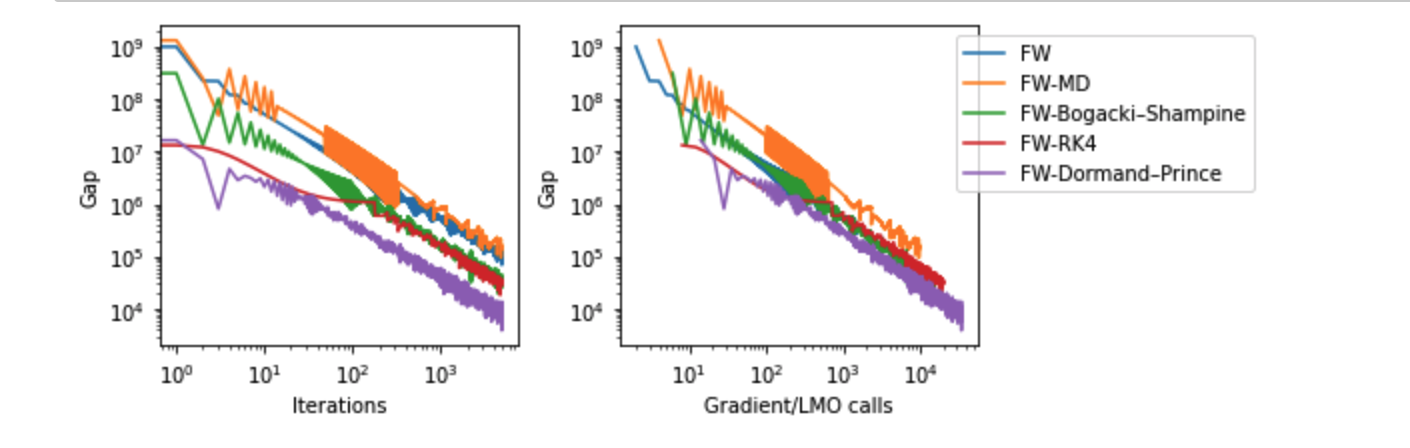}
  \caption{\textbf{Compressed sensing.} $500$ samples, $100$ features, 10\% sparsity ground truth, $\alpha = 1000$. }
  \label{fig:more_quadratic}
\end{figure}

\end{document}